\documentclass[12pt]{amsart}
\usepackage{amsmath, color}
\usepackage{amssymb}
\usepackage{amsfonts}
\usepackage{amsthm}
\usepackage{mathrsfs}
\usepackage[all]{xy}

\newcommand{\be}{\begin{equation}}
\newcommand{\ee}{\end{equation}}
\newtheorem{theorem}{Theorem}[section]
\newtheorem{lemma}[theorem]{Lemma}

\newtheorem{proposition}[theorem]{Proposition}

\theoremstyle{definition}
\newtheorem{definition}[theorem]{Definition}
\newtheorem{example}[theorem]{Example}

\theoremstyle{remark}
\newtheorem{remark}[theorem]{Remark}

\theoremstyle{conjecture}

\numberwithin{equation}{section}

\begin{document}

\title{R-matrix realization of two-parameter quantum group $U_{r,s}(\mathfrak{gl}_n)$}
\author{Naihuan Jing, Ming Liu$^*$}
\address{NJ: Department of Mathematics, North Carolina State University, Raleigh, NC 27695, USA
and School of Mathematical Sciences, South China University of Technology, Guangzhou 510640, China}
\address{ML: Chern Institute of Mathematics, Nankai University, Tianjin 30071, China}

\thanks{{\scriptsize
\hskip -0.4 true cm MSC (2010): Primary: 17B37; Secondary: 20G42, 16T25.
\newline Keywords: Quantum groups, determinants, Casimir elements, Yang-Baxter equations\\
$*$Corresponding author.
}}

\maketitle

\begin{abstract} We provide a Faddeev-Reshetikhin-Takhtajan's RTT approach to the quantum group
$Fun(GL_{r,s}(n))$ and the quantum enveloping algebra $U_{r,s}(\mathfrak{gl}_n)$
corresponding to the two-parameter $R$-matrix. We prove
that the quantum determinant $det_{r,s}T$ is a quasi-central element in $Fun(GL_{r,s}(n))$
generalizing earlier results of Dipper-Donkin and Du-Parshall-Wang. The explicit
formulation provides an interpretation of the deforming parameters,
and the quantized algebra $U_{r,s}(R)$
is identified to $U_{r,s}(\mathfrak{gl}_n)$ as the dual algebra. We then
construct $n-1$ quasi-central elements in $U_{r,s}(R)$ which are
analogues of higher Casimir elements in $U_q(\mathfrak{gl}_n)$.

\end{abstract}

\section{Introduction}
Quantum groups were discovered as certain noncommutative and noncocommutative Hopf algebras by
Drinfeld \cite{D3} and Jimbo \cite{J}. The standard definition of a quantum group $U_q(\mathfrak{g})$ is given as a $q$-deformation
of universal enveloping algebra of a simple Lie algebra $\mathfrak{g}$ generated by the Chevalley generators under
the Serre relations based on the data coming from the corresponding Cartan matrix.

Faddeev,
Reshetikhin and Takhtajan \cite{FRT} gave another realization of the
quantum groups using the solutions $R$ of the
Yang-Baxter equation:
\begin{equation}\label{YB Eq}
R_{12}R_{13}R_{23}=R_{23}R_{13}R_{12}.
\end{equation}
They also studied the quantum function algebra $Fun(GL_{q}(n))$ for Lie group $GL_n$ in \cite{FRT}.
In particular, the quantum determinant $qdet(T)$ for the quantum function algebra $Fun(GL_{q}(n))$ was introduced and proved to be a special central element.
Furthermore, the authors studied the algebra $U(R)$ as a dual algebra of the quantum
function algebra and proved that $U(R)$ is isomorphic to $U_q(\mathfrak{g})$. This approach can systematically provide a complete
set of generators for the center of the quantum enveloping algebra $U_q(\mathfrak{sl}_n)$:
\begin{equation}
c_k=\sum_{\sigma,\sigma'\in S_n}(-q)^{l(\sigma)+l(\sigma')}l^{+}_{\sigma(1),\sigma'(1)}...l^{+}_{\sigma(k),\sigma'(k)}
l^{-}_{\sigma(k+1),\sigma'(k+1)}...l^{-}_{\sigma(n),\sigma'(n)},
\end{equation}
where $l^{\pm}_{ij}$ are the $q$-analogs of the Weyl root vectors corresponding to the roots
$\pm(\epsilon_i-\epsilon_j)$.

Two-parameter general linear and special linear quantum groups were introduced by Takeuchi \cite{T} in 1990. A special case of the quantum coordinate algebra
 was the subject of study in relationship with the Schur
$q$-algebra \cite{DD, DPW} where a quantum determinant
was shown to be quasi-central. The two-parameter quantum enveloping algebras have later gained full attention after
 Benkart and Witherspoon's works \cite{BW1, BW2} on the $(r,s)$-deformed quantum algebras associated with
 $\mathfrak{gl}_n$ and $\mathfrak{sl}_n$, where the quantum R-matrix and
the Drinfeld doubles were obtained (see also \cite{BGH}). They further showed that the representation theory
of the two-parameter quantum enveloping algebras can be similarly developed as
the one-parameter case. In this context they established a two-parameter quantum Schur-Weyl duality
(see also \cite{JL}). In \cite{JZ} an interpretation of
the two parameters was demonstrated through a combinatorial
realization of the affine analog \cite{HRZ} of the quantum linear algebra.

The quantum inverse scattering method has been very successful in studying various quantum groups \cite{CP}
such as the Yangian algebras \cite{MNO, M} and finite W-algebras \cite{BR}. For quantum matrix algebras
the method has been particularly useful \cite{FRT, HIOPT}.
In this paper we will generalize the Faddeev-Reshetikhin-Takhatajan (FRT)
approach \cite{FRT} to the two-parameter case by studying the two-parameter
quantum linear groups and the quantum general/special linear algebras as
dual Hopf algebras. The anti-symmetric tensors \cite{JL} associated to
the R-matrix \cite{BW3} help one to define 2-parameter quantum determinant
$det_{r,s}T$ and the quantum Casimir elements.
If $t_{ij}$ are the generators
of the 2-parameter quantum monoid $Mat_{r,s}(n)$, Using
the FRT method we
obtain that the row determinant and the column determinant are given respectively by
\begin{align*}
det_{r,s}T&=\sum_{\sigma\in Sn}(-s)^{l(\sigma)}t_{\sigma(1),1}\cdots t_{\sigma(n),n}\\
&=\sum_{\sigma\in Sn}(-r)^{-l(\sigma)}t_{1, \sigma(1)}\cdots t_{n, \sigma(n)}.
\end{align*}
This means that the two parameters $r, s$ exactly correspond to column permutations and row
permutations respectively.

 By a similar method to FRT approach, we are able to
define naturally and systematically the Weyl root vectors or Gauss generators for $U_{r,s}(\mathfrak{sl}_n)$
and their commutation relations can be compactly written as matrix equations
in terms of the 2-parameter R-matrices. Moreover we can obtain the explicit
formulae for the higher Casimiar elements and show that they are
quasi-central. In studying representation theory of 2-parameter quantum algebras,
it seems to be more naturally to use the quasi-commutative Casimir elements
instead of using the central elements given the similarity of the
two representation theories.

The paper is organized as follows. In section 2
we study the two-parameter quantum function algebra $Fun(GL_{r,s}(n))$ and its two-parameter
quantum determinant. Different from the one-parameter case, we prove that
the two-parameter
quantum determinant $det_{r,s}(T)$ is not a central element but a quasi-central element.
In section 3, we study the algebra $U(R)$ as the dual of $Fun(GL_{r,s}(n))$,
and analogue to the one-parameter case, we give $n-1$ quasi-central elements $c_k$ in $U(R)$
which generalize the classical higher Casimir elements in the center of the universal enveloping algebra.
In section 4, we study the commutation relations between the Cartan-Weyl generators of $U(R)$ by using the Gauss decomposition of the matrix $L^{\pm}$
inside $U(R)$.
In section 5, we show that $U_{r,s}(\mathfrak{gl}_n)$ in Drinfeld-Jimbo realization
is isomorphic to $U(R)$ in FRT realization.

\section{Two-parameter quantum algebra $Fun(GL_{r,s}(n))$}

In this section we study the two-parameter quantum ccordinate
algebra $Fun(GL_{r,s}(n))$ associated to the general linear group $GL_n$
using the FRT method.

Let $V=\mathbb C^n$, the $n$-dimensional complex space with the basis
of column vectors $e^i=(0, \cdots, 1, \cdots, 0)^T$.
Let $e_{ij}$ be the unit matrices acting on $V$ so that $e_{ij}e^k=\delta_{jk}e^i$. The
two-parameter quantum $R$-matrix $\in End(V\otimes V)$ is given as
\begin{equation}\label{R-matrix}
\begin{aligned}
R&=\sum_{ijkl}R_{kl}^{ij}e_{ik}\otimes e_{jl}\\
&=s\sum_{i=1}^{n}e_{ii}\otimes e_{ii}+rs\sum_{i>j}e_{ii}\otimes e_{jj}+
\sum_{i<j}e_{ii}\otimes
e_{jj}+(s-r)\sum_{i>j}e_{ij}\otimes e_{ji},
\end{aligned}
\end{equation}
which satisfies the well-known Yang-Baxter equation:
\begin{equation*}
R_{12}R_{13}R_{23}=R_{23}R_{13}R_{12},
\end{equation*}
where $R_{12}=R\otimes 1$, and $R_{ij}$ acts on the $i$th and $j$th
copies of $V$ inside $V^{\otimes 3}$.
\begin{remark}
When $s=q=r^{-1}$, the R-matrix (\ref{R-matrix}) is the one-parameter R-matrix considered in \cite{FRT}.
We remark that the R-matrix can not be re-scaled to a one-parameter R-matrix.
\end{remark}

We set
\begin{equation}\label{braided R-matrix}
\widehat{R}=PR=s\sum_{i=1}^{n}e_{ii}\otimes e_{ii}+rs\sum_{i<j}e_{ij}\otimes e_{ji}+
\sum_{i>j}e_{ij}\otimes
e_{ji}+(s-r)\sum_{i<j}e_{ii}\otimes e_{jj},
\end{equation}
where $P$ is the permutation operator $P(u\otimes v)=v\otimes u$. Then the Yang-Baxter equation (\ref{YB Eq}) is equivalent to the braid relation
\begin{equation}\label{e:braid}
\widehat{R}_{12}\widehat{R}_{23}\widehat{R}_{12}=\widehat{R}_{23}\widehat{R}_{12}\widehat{R}_{23}.
\end{equation}
Furthermore $\widehat{R}$ satisfies the following Hecke relation:
$$(\widehat{R}-s)(\widehat{R}+r)=0.$$
For $r+s\neq 0$, $\widehat{R}$ has the following spectral
decomposition
$$\widehat{R}=sP^{+}-rP^{-},$$
where $P^{\pm}$ are the idempotents such that $P^+P^-=P^-P^+=0, P^++P^-=1$ and they are given by
\begin{align*}
P^{+}&=\frac{\widehat{R}+r}{r+s},\\
P^{-}&=\frac{-\widehat{R}+s}{r+s}.
\end{align*}

Following \cite{FRT} we define the quantum algebra $A(R)$ of the matrix monoid for the two-parameter R-matrix  (\ref{R-matrix}).
\begin{definition} The algebra
$A(R)$ is an associative algebra generated by $t_{ij}$,
$1\leq i,j\leq n$ subject to the quadratic relations defined by
\begin{equation}\label{defining relations of A(R)}
RT_{1}T_2=T_2T_{1}R,
\end{equation}
where $T_{1}=\sum_{ij}t_{ij}\otimes e_{ij}\otimes 1$, $T_2=\sum_{ij}t_{ij}\otimes 1\otimes e_{ij}
\in A(R)\otimes End(V^{\otimes 2})$.
\end{definition}

A representation of $A(R)$ is a linear map $T\longrightarrow A\in End(W)$ such that
$RA_1A_2=A_2A_1R$ on $End(W\otimes V^{\otimes 2})$. 

\begin{proposition}\cite{FRT}
The algebra $A(R)$ is a bialgebra with comultiplication
\begin{equation*}
\Delta(t_{ij})=\sum_{k} t_{ik}\otimes t_{kj},~~~~~~~~~~~~i,j=1,...,n,
\end{equation*}
and co-unit $\epsilon$
\begin{equation*}
\epsilon(t_{ij})=\delta_{ij},~~~~~~~~~~~~~i,j=1,...,n.
\end{equation*}
\end{proposition}

Suppose there are elements $t'_{ij}$, $1\leq i, j\leq n$ such that
\begin{equation*}
\sum_{k}t'_{ik}t_{kj}=\sum_{k}t_{ik}t'_{kj}=\delta_{ij}\cdot 1.
\end{equation*}
Then the algebra $Fun(GL_{r,s}(n))$ is then defined to be the associative
algebra generated by $t_{ij}, t'_{ij}$ subject to aforementioned relations involving
the generators. Clearly one has $T'T=TT'=1$ in $Fun(GL_{r,s}(n))$.
Moreover $Fun(GL_{r,s}(n))$ is a Hopf algebra with the antipode given by
$S(t_{ij})=t'_{ij}$. We remark that the elements $t'_{ij}$ will be
shown to exist once we prove that the quantum determinant is a
regular element in the ring $Fun(GL_{r,s}(n))$ as the elements
$t_{ij}'$ can be solved by the quantum Cramer rule.

\begin{example} The two-parameter quantum group $Fun(GL_{r,s}(2))$
is generated by $t_{ij}, det^{\pm 1}_{r,s}$ subject to the relations
\begin{align*}
t_{11}t_{12}&=r^{-1}t_{12}t_{11}, \quad t_{11}t_{21}=st_{21}t_{11}\\
t_{21}t_{22}&=r^{-1}t_{21}t_{22}, \quad t_{12}t_{22}=st_{22}t_{12}\\
t_{12}t_{21}&=rst_{21}t_{12},\\
t_{11}t_{22}&-t_{22}t_{11}=(s-r)t_{21}t_{12},\\
det_{r,s}det_{r,s}^{-1}&=det_{r,s}^{-1}det_{r,s}=1,\\
(det_{r,s})t_{ij}&=(rs)^{i-j}t_{ij}(det_{r,s}),
\end{align*}
where
\begin{align*}
det_{r,s}&=t_{11}t_{22}-st_{12}t_{21}=t_{22}t_{11}-rt_{12}t_{21}\\
&=t_{11}t_{22}-r^{-1}t_{12}t_{21}.
\end{align*}
\end{example}

We now turn to the general quantum determinant of $Fun(GL_{r,s}(n))$ using the quantum inverse scattering
method via anti-symmetric tensors (cf. \cite{JL}).

Let $V^*$ be the dual
space of $V$ spanned by the row vectors $e_i=e_i^T$.
The endomorphism space $End(V)$ is identified with $V\otimes V^*$,
which has the natural basis elements $e_{ij}=e^ie_j:=e^i\otimes e_j$.

Before introducing the quantum determinant, we introduce the $(r,s)$-deformed antisymmetric
tensors $|\varepsilon\rangle\doteq|\varepsilon_n\rangle\in V^{\otimes n}$, $\langle\varepsilon|\doteq\langle\varepsilon_n|\in {(V^*)}^{\otimes n}$
as follows:
\begin{align}\label{e:normalization}
\langle\varepsilon|\varepsilon\rangle&=1,\\
\label{Anti-Sym1}
\langle\varepsilon|P^{+}_{k,k+1}&=0,~~~~~~~~~~~~~~k=1,2,...,n-1,\\
\label{Anti-Sym2}
P^{+}_{k,k+1}|\varepsilon\rangle&=0,~~~~~~~~~~~~~~~k=1,2,...,n-1.
\end{align}
where $P^+_{k, k+1}=1^{\otimes(k-1)}\otimes P^+\otimes 1^{\otimes(n-k-1)}$ with $P^{+}$
at the $(k, k+1)$-position.
It is straightforward to prove that the above equations have a unique solution
up to normalization, and the $(r,s)$-deformed antisymmetric tensors
can be chosen as follows:
\begin{equation}
\langle\varepsilon|=\frac{1}{[n]_{r,s}!}\sum_{\sigma\in S_n}(-s)^{l(\sigma)}e_{\sigma(1)}\otimes\cdots \otimes e_{\sigma(n)},
\end{equation}

\begin{equation}
|\varepsilon\rangle=r^{\frac{n(n-1)}{2}}\sum_{\sigma\in S_n}(-r)^{-l(\sigma)}e^{\sigma(1)}\otimes\cdots \otimes e^{\sigma(n)},
\end{equation}
where, $[k]_{r,s}=\frac{s^k-r^k}{s-r}$, $[k]_{r,s}!=[1]_{r,s}[2]_{r,s}\cdots [k]_{r,s}$, and $e_i$ the i-th row vector, $e^j$ the j-th
column vector.

In view of the RTT relations (\ref{defining relations of A(R)}) of $A(R)$ and definition of $P^{+}$,
we see that $\langle\varepsilon|T_1\cdots T_n$ and $T_1\cdots T_n|\varepsilon\rangle$
satisfy the following equations respectively:
\begin{align*}
\langle\varepsilon|T_1\cdots T_nP^{+}_{k,k+1}&=0,~~~~~~~~~~~~~~k=1,2,...,n-1,\\
P^{+}_{k,k+1}T_1\cdots T_n|\varepsilon\rangle&=0,~~~~~~~~~~~~~~~k=1,2,...,n-1.
\end{align*}

From the uniqueness of the solutions to Eq. (\ref{Anti-Sym1})--(\ref{Anti-Sym2}) it follows that
\begin{align}\label{Eq_5.7}
\langle\varepsilon|T_1\cdots T_n&=c_T\langle\varepsilon|,\\
\label{Eq_5.8}
T_1\cdots T_n|\varepsilon\rangle&=c^T|\varepsilon\rangle,
\end{align}
where $c_T, c^T \in Fun_{r,s}(GL_n)$.
Subsequently it follows from (\ref{e:normalization}) that
\begin{equation}\label{e:det}
c_T=c^T=\langle\varepsilon|T_1\cdots T_n|\varepsilon\rangle.
\end{equation}

\begin{definition} The quantum determinant of $Fun(GL_{r,s}(n))$ is defined to be the
normalization factor
$$det_{r,s}T=\langle\varepsilon|T_1\cdots T_n|\varepsilon\rangle.$$
\end{definition}

To find an explicit formula of $det_{r,s}T$ we
consider the rank one tensor $A_n=|\varepsilon\rangle\langle\varepsilon|$
acting on $(\mathbb{C}^n)^{\otimes n}$. It is easily seen that
$A_n^2=A_n$. The following result gives some compact form
of various commutation relations in $Fun(GL_{r,s}(n))$.
\begin{proposition}\label{prop_5.2}
$A_nT_1T_2\cdots T_n=T_1T_2\cdots T_nA_n=A_nT_1T_2\cdots T_nA_n=det_{r,s}TA_n$.
\end{proposition}
\begin{proof} By definition
$T$ commutes with $|\varepsilon\rangle$ or $\langle\varepsilon|$. It follows from
$A_n^2=A_n$ that
\begin{align*}
A_nT_1\cdots T_n&=A_nT_1\cdots T_nA_n\\
&=|\varepsilon\rangle\langle\varepsilon|T_1\cdots T_n|\varepsilon\rangle\langle\varepsilon|\\
&=(\langle\varepsilon|T_1\cdots T_n|\varepsilon\rangle)|\varepsilon\rangle\langle\varepsilon|\\
&=det_{r,s}TA_n.
\end{align*}
The other identities are proved similarly.
\end{proof}

\begin{remark}\label{r:trace} The quantum determinant is also expressed as a partial trace.
\begin{align*}
Tr_{1,\ldots,n}(A_nT_1\cdots T_n)&=Tr_{1,\ldots,n}(|\varepsilon\rangle\langle\varepsilon|T_1\cdots T_n)\\
&=Tr_{1,\ldots,n}(\langle\varepsilon|T_1\cdots T_n)|\varepsilon\rangle)\\
&=det_{r,s}T.
\end{align*}
\end{remark}

In the following proposition, we give the explicit expression of the two-parameter quantum determinant $det_{r,s}T$
in $Fun(GL_{r,s}(n))$.
\begin{proposition} For any fixed $\eta\in S_n$ we have
\begin{equation}\label{e:det0}
\begin{aligned}
det_{r,s}T&=&(-s)^{-l(\eta)}\sum_{\sigma\in Sn}(-s)^{l(\sigma)}t_{\sigma(1),\eta(1)}\cdots t_{\sigma(n),\eta(n)}\\
&=&(-r)^{l(\eta)}\sum_{\sigma\in Sn}(-r)^{-l(\sigma)}t_{\eta(1),\sigma(1)}\cdots t_{\eta(n),\sigma(n)}.
\end{aligned}
\end{equation}
In particular,

\begin{equation}
\begin{aligned}
det_{r,s}T&=\sum_{\sigma\in Sn}(-s)^{l(\sigma)}t_{\sigma(1),1}\cdots t_{\sigma(n),n}\\
&=\sum_{\sigma\in Sn}(-r)^{-l(\sigma)}t_{1,\sigma(1)}\cdots t_{n,\sigma(n)}.
\end{aligned}
\end{equation}
\end{proposition}

\begin{proof} The two formulas of (\ref{e:det0}) are proved similarly, so we only consider the
first one. Recall that Proposition \ref{prop_5.2} says that $A_nT_1\cdots T_n=det_{r,s}T A_n$.
For any $\eta\in S_n$ we apply the left-hand side to the column
vector $e^{\eta(1)}\otimes\cdots \otimes e^{\eta(n)}$:
\begin{equation}\label{e:det2}
\begin{aligned}
A_nT_1\cdots T_n (e^{\eta(1)}\otimes\cdots \otimes e^{\eta(n)})&=
A_n\sum_{i_1,...,i_n}t_{i_1,\eta(1)}\cdots t_{i_n,\eta(n)}(e^{i_1}\otimes\cdots \otimes e^{i_n})\\
&=\sum_{i_1,...,i_n}t_{i_1,\eta(1)}\cdots t_{i_n,\eta(n)}|\varepsilon\rangle\langle\varepsilon|(e^{i_1}\otimes\cdots \otimes e^{i_n})\\
&=\frac{1}{[n]_{r,s}!}\sum_{\sigma\in Sn}(-s)^{l(\sigma)}t_{\sigma(1),\eta(1)}\cdots t_{\sigma(n),\eta(n)}|\varepsilon\rangle.
\end{aligned}
\end{equation}
On the other hand we have
\begin{equation}\label{e:det3}
det_{r,s}TA_n(e^{\eta(1)}\otimes\cdots \otimes e^{\eta(n)})=\frac{1}{[n]_{r,s}!}(-s)^{l(\eta)}det_{r,s}T|\varepsilon\rangle.
\end{equation}
Paring with $\langle\varepsilon|$ to the right-hand sides of (\ref{e:det2}-\ref{e:det3}) we obtain that
\begin{equation*}
det_{r,s}T=(-s)^{-l(\eta)}\sum_{\sigma\in Sn}(-s)^{l(\sigma)}t_{\sigma(1),\eta(1)}\cdots t_{\sigma(n),\eta(n)}.
\end{equation*}
\end{proof}
\begin{remark} One can also show the following
\begin{equation}
\begin{aligned}
det_{r,s}T&=\sum_{\sigma\in Sn}(-r)^{l(\sigma)}t_{\sigma(n),n}\cdots t_{\sigma(1),1}\\
&=\sum_{\sigma\in Sn}(-s)^{-l(\sigma)}t_{n,\sigma(n)}\cdots t_{1,\sigma(1)}.
\end{aligned}
\end{equation}
When $s=q=r^{-1}$, the two-parameter quantum determinant is reduced to the one-parameter quantum determinant
$$det_{q}T=\sum_{\sigma\in Sn}(-q)^{l(\sigma)}t_{1,\sigma(1)}\cdots t_{n,\sigma(n)}.
$$
\end{remark}
Different from the one-parameter case, the two-parameter quantum determinant is not a central
element, but a quasi-cental element. It is still a regular element
in the ring $Fun(GL_{r,s}(n))$, so the antipode is well-defined.
\begin{theorem}\label{t:det}
The two-parameter quantum determinant $det_{rs}T$ is quasi-central in
$Fun(GL_{r,s}(n))$. In fact,
\begin{equation}
(det_{r,s}T )T=M^{-1}TM(det_{r,s}T),
\end{equation}
where
$$M=\begin{pmatrix}
    (rs)^{n-1} &  &  &  \\
     & (rs)^{n-2} &  &  \\
     &  & \ddots &  \\
     &  &  & 1\\
  \end{pmatrix},
$$
which implies that $(det_{r,s}T)t_{ij}=(rs)^{i-j}t_{ij}(det_{r,s}T)$.
\end{theorem}
Before we prove the theorem, we need the following lemma.

\begin{lemma}\label{lemma5.5} Let $M_{n+1}=1^{\otimes n}\otimes M\in End(V)^{\otimes(n+1)}$, then
$sM_{n+1}(A_n\otimes 1)=(A_n\otimes 1)R_{1,n+1}\cdots R_{n,n+1}$ on $V^{\otimes(n+1)}$.
\end{lemma}

\begin{proof} We order the indices $(ij)$ lexicographically. As $R_{ij}^{ji}=s-r$ when
$i>j$, the matrix $R$ is an upper triangular block matrix with diagonal blocks given by
\begin{equation}\label{Eq_5.15}
R^{ik}_{ik}=\left\{
                   \begin{array}{ll}
                     s, & \hbox{i=k;} \\
                     rs, & \hbox{$i>k$;} \\
                     1, & \hbox{$i<k$.}
                   \end{array}
                 \right.
\end{equation}
In the same manner $R_{1,n+1}R_{2,n+1}\cdots R_{n,n+1}$ is also an
upper triangular block matrix with diagonal blocks given by
\begin{equation}\label{e:diagonal-block}
\sum_{i_1,...,i_n,k}R^{i_1k}_{i_1k}\cdots R^{i_nk}_{i_nk}e_{i_1,i_1}\otimes \cdots \otimes e_{i_n,i_n}\otimes e_{k,k}.
\end{equation}
We then get that for $\langle\varepsilon|=\langle\varepsilon|\otimes 1$
\begin{equation}
\langle\varepsilon|R_{1,n+1}\cdots R_{n,n+1}|\varepsilon\rangle=
\sum_{k=1}^n \frac{r^{\frac{n(n-1)}{2}}}{[n]_{r,s}!}\sum_{\sigma\in S_n}(sr^{-1})^{l(\sigma)}R^{\sigma(1)k}_{\sigma(1)k}\cdots R^{\sigma(n)k}_{\sigma(n)k}(e_{k,k})_{n+1}.
\end{equation}
Moreover, by Eq. (\ref{Eq_5.15}) we have
$R^{\sigma(1)k}_{\sigma(1)k}\cdots R^{\sigma(n)k}_{\sigma(n)k}=s(rs)^{n-k}$ for any $\sigma\in S_n$.
Therefore
\begin{equation}
\langle\varepsilon|R_{1,n+1}\cdots R_{n,n+1}|\varepsilon\rangle=
\sum_{k=1}^n s(rs)^{n-k}(e_{k,k})_{n+1}=sM_{n+1}.
\end{equation}

Note that $A_{n}=|\varepsilon\rangle\langle\varepsilon|$, we then obtain that
\begin{equation}
(A_n\otimes 1)R_{1,n+1}\cdots R_{n,n+1}(A_n\otimes 1)=sM_{n+1}(A_n\otimes 1).
\end{equation}
On the other hand the map
$$T_i\mapsto R_{i,n+1}$$ defines a representation of $Fun(GL_{r,s}(n))$, so $A_nT_1\cdots T_n=T_1\cdots T_nA_n$
implies that
\begin{align*}
(A_n\otimes 1)R_{1,n+1}\cdots R_{n,n+1}(A_n\otimes 1)&=(A_n\otimes 1)^2R_{1,n+1}\cdots R_{n,n+1}\\
&=(A_n\otimes 1)R_{1,n+1}\cdots R_{n,n+1}=sM_{n+1}(A_n\otimes 1).
\end{align*}
\end{proof}

Now we prove Theorem \ref{t:det}.
\begin{proof} For brevity we write $A_n$ for
$(A_n\otimes 1)$ which is clear from the context.
Using the RTT relations of $Fun(GL_{r,s}(n))$ to move
$T_{n+1}$, we have that
\begin{equation}
A_nT_1\cdots T_nT_{n+1}=A_n(R_{1,n+1}\cdots R_{n,n+1})^{-1}T_{n+1}T_1\cdots T_n(R_{1,n+1}\cdots R_{n,n+1}).
\end{equation}
It follows from Lemma \ref{lemma5.5} that
\begin{equation}
A_nT_1\cdots T_nT_{n+1}=s^{-1}M_{n+1}^{-1}T_{n+1}A_nT_1\cdots T_n(R_{1,n+1}\cdots R_{n,n+1}).
\end{equation}
Note that $A_nT_1\cdots T_n=A_nT_1\cdots T_nA_n$ (Prop. \ref{prop_5.2}). Applying Lemma \ref{lemma5.5} again
we obtain that
\begin{align*}
A_nT_1\cdots T_nT_{n+1}&=s^{-1}M_{n+1}^{-1}T_{n+1}A_nT_1\cdots T_nA_n(R_{1,n+1}\cdots R_{n,n+1})\\
&=M_{n+1}^{-1}T_{n+1}M_{n+1}A_nT_1\cdots T_n.
\end{align*}
Taking partial trace $Tr_{1,\ldots,n}$ (see Remark \ref{r:trace}), 
we finally get that $(det_{r,s}T) T=M^{-1}TM (det_{r,s}T)$.
\end{proof}

\section{FRT realization of two-parameter quantum groups}

In this section we study the algebra $U(R)$
as the dual Hopf algebra of $Fun(GL_{r,s}(n))$.

Consider the R-matrix $\in End(V\otimes V)$
$$R^{(+)}=PRP=R_{21}=s\sum_{i=1}^{n}e_{ii}\otimes e_{ii}+rs\sum_{i<j}e_{ii}\otimes e_{jj}+
\sum_{i>j}e_{ii}\otimes
e_{jj}+(s-r)\sum_{i<j}e_{ij}\otimes e_{ji},$$
which is another solution of the Yang-Baxter equation (\ref{YB Eq}).

\begin{definition}
$U(R)$ is an associative algebra with generators  $l^{+}_{ij}$, $l^{-}_{ji}$, $1\leq i\leq j\leq n$
subject to the quadratic relations given by
\begin{align}\label{generating relation1}
R^{(+)}L^{\pm}_{1}L^{\pm}_2&=L^{\pm}_2L^{\pm}_{1}R^{(+)},\\
\label{generating relation2}
R^{(+)}L^{+}_1L^{-}_2&=L^{-}_2L^{+}_1R^{(+)},
\end{align}
where $L^{\pm}_1=\sum l^{\pm}_{ij}e_{ij}\otimes 1$, $L^{\pm}_2=\sum l^{\pm}_{ij}1\otimes e_{ij}$ and $L^{\pm}=(l^{\pm}_{ij})$ $(1\leq i,j\leq n)$ are invertible triangular matrices
with $l^{+}_{ij}=l^{-}_{ji}=0$ for $1\leq j<i\leq n$.
\end{definition}

\begin{proposition}
The algebra $U(R)$ is a Hopf algebra with comultiplication, antipode and counit given by
\begin{align}
\Delta(l_{ij}^{\pm})&=\sum_{k}l_{ik}^{\pm}\otimes l_{kj}^{\pm},\\
S(L^{\pm})&=(L^{\pm})^{-1},\\
\epsilon(l_{ij}^{\pm})&=\delta_{ij}.
\end{align}
\end{proposition}

As in the dual algebra case, let $\widehat{R}^{(+)}=PR^{(+)}=\widehat{R}_{21}$, then
$$\widehat{R}^{(+)}=s\sum_{i=1}^{n}e_{ii}\otimes e_{ii}+rs\sum_{i>j}e_{ij}\otimes e_{ji}+
\sum_{i<j}e_{ij}\otimes
e_{ji}+(s-r)\sum_{i>j}e_{ii}\otimes e_{jj}.$$
Moreover it satisfies the Hecke relation
\begin{equation}
(\hat{R}^{(+)}-s)(\hat{R}^{(+)}+r)=0.
\end{equation}

For this R-matrix $\widehat{R}^{(+)}$, we introduce antisymmetric tensors $\langle\varepsilon^{(+)}|$, $|\varepsilon^{(+)}\rangle$ by the equations:
\begin{align}\label{e:antisym1}
\langle\varepsilon^{(+)}|\varepsilon^{(+)}\rangle &=1,\\ \label{e:antisym2}
\langle\varepsilon^{(+)}|(\hat{R}^{(+)}_{i,i+1}+r)&=0,~~~~~~~~~~~~~~~~~~~i=1,2,\ldots,n-1,\\ \label{e:antisym3}
(\hat{R}^{(+)}_{i,i+1}+r)|\varepsilon^{(+)}\rangle &=0,~~~~~~~~~~~~~~~~~~~i=1,2,\ldots,n-1.
\end{align}
These equations determine the antisymmetric tensors
up to a constant. It is easy to see that $\langle\varepsilon^{(+)}|$, $|\varepsilon^{(+)}\rangle$
can be chosen as follows.
\begin{align}
\langle\varepsilon^{(+)}|&=\frac{1}{[n]_{r,s}!}\sum_{\sigma\in S_n}(-s)^{-l(\sigma)}e_{\sigma(1)}\otimes\cdots \otimes e_{\sigma(n)}\\
|\varepsilon^{(+)}\rangle &=s^{\frac{n(n-1)}{2}}\sum_{\sigma\in S_n}(-r)^{l(\sigma)}e^{\sigma(1)}\otimes\cdots \otimes e^{\sigma(n)}.
\end{align}

Define the rank one matrix $A^{(+)}_n=|\varepsilon^{(+)}\rangle\langle\varepsilon^{(+)}|$.
Explicitly we have that
\begin{equation}
A^{(+)}_n=
\frac{s^{\frac{n(n-1)}{2}}}{[n]_{r,s}!}\sum_{\sigma,\tau\in S_n}(-r)^{l(\tau)}(-s)^{-l(\sigma)}
e_{\tau(1)\sigma(1)}\otimes\cdots \otimes e_{\tau(n)\sigma(n)}.
\end{equation}
The following result gives commutation relations among the Weyl generators of $U(R)$.
\begin{proposition}\label{prop3.4} In $U(R)$ the following identities are satisfied for $k=1,2,\ldots,n$:
\begin{align*}
\begin{aligned}
A^{(+)}_nL_1^{+}\cdots L_k^{+}L_{k+1}^{-}\cdots L_n^{-}A^{(+)}_n&=&A^{(+)}_nL_1^{+}\cdots L_k^{+}L_{k+1}^{-}\cdots L_n^{-}\\
&=&L_1^{+}\cdots L_k^{+}L_{k+1}^{-}\cdots L_n^{-}A^{(+)}_n.\\
\end{aligned}
\end{align*}
\end{proposition}
\begin{proof} Using the RTT defining relations (\ref{generating relation1}-\ref{generating relation2}),
we have that
\begin{align*}
\langle\varepsilon^{(+)}|L_1^{+}\cdots L_k^{+}L_{k+1}^{-}\cdots L_n^{-}(\hat{R}^{(+)}_{i,i+1}+r)&=\langle\varepsilon^{(+)}|(\hat{R}^{(+)}_{i,i+1}+r)L_1^{+}\cdots L_k^{+}L_{k+1}^{-}\cdots L_n^{-},\\
(\hat{R}^{(+)}_{i,i+1}+r)L_1^{+}\cdots L_k^{+}L_{k+1}^{-}\cdots L_n^{-}|\varepsilon^{(+)}\rangle &=L_1^{+}\cdots L_k^{+}L_{k+1}^{-}\cdots L_n^{-}(\hat{R}^{(+)}_{i,i+1}+r)|\varepsilon^{(+)}\rangle
\end{align*}
for $k=1,2,\ldots,n-1$.

By the uniqueness of the solution to (\ref{e:antisym2}-\ref{e:antisym3}) we can assume that
\begin{align*}
\langle\varepsilon^{(+)}|L_1^{+}\cdots L_k^{+}L_{k+1}^{-}\cdots L_n^{-}&=a_k\langle\varepsilon^{(+)}|,\\
L_1^{+}\cdots L_k^{+}L_{k+1}^{-}\cdots L_n^{-}|\varepsilon^{(+)}\rangle &=b_k|\varepsilon^{(+)}\rangle,
\end{align*}
where $a_k, b_k\in U(R)$.
By the normalization (\ref{e:antisym1}) it follows that
\begin{align*}
a_k=b_k=\langle\varepsilon^{(+)}|L_1^{+}\cdots L_k^{+}L_{k+1}^{-}\cdots L_n^{-}|\varepsilon^{(+)}\rangle.
\end{align*}
Subsequently we get
$$A_n^{(+)}L_1^{+}\cdots L_k^{+}L_{k+1}^{-}\cdots L_n^{-}=L_1^{+}\cdots L_k^{+}L_{k+1}^{-}\cdots L_n^{-}A_n^{(+)}=a_kA_n^{(+)}.$$
\end{proof}

The following lemma will be needed to compute a partial trace in Theorem \ref{Thm3.6}.
\begin{lemma}\label{lemma3.5} In $End(V)^{\otimes(n+1)}$ one has that
\begin{align}
(A^{(+)}_n\otimes 1)R^{(+)}_{1,n+1}\cdots R^{(+)}_{n,n+1}&=sM'_{n+1}(A^{(+)}_n\otimes 1),\\
R^{(+)}_{0,n}\cdots R^{(+)}_{0,1}(1\otimes A^{(+)}_n) &=sM_0(1\otimes A^{(+)}_n),
\end{align}
where the indices of $V$ in $V^{\otimes(n+1)}$ in the first (resp. second) identity are $1, \ldots, n+1$ 
(resp. $0, 1, \ldots, n$) and $R_{ij}$ and $M_{n+1}$ (resp. $M_0'$) are defined accordingly. Here 
 \begin{equation}
 M'=
 \begin{pmatrix}
   1 &  &  &  \\
    & rs &  &  \\
    &  & \ddots &  \\
    &  &  & (rs)^{n-1} \\
 \end{pmatrix},
\end{equation}
and
 \begin{equation}
 M=
 \begin{pmatrix}
   (rs)^{n-1} &  &  &  \\
    & (rs)^{n-2} &  &  \\
    &  & \ddots &  \\
    &  &  & 1 \\
 \end{pmatrix}.
\end{equation}

\end{lemma}

\begin{proof} These two identities are proved similarly as Lemme \ref{lemma5.5}. Note that $R^{(+)}$
is also a triangular block matrix with diagonal blocks given by
\begin{equation}
(R^{(+)})_{ik}^{ik}=\left\{
                   \begin{array}{ll}
                     s, & \hbox{i=k;} \\
                     rs, & \hbox{$i<k$;} \\
                     1, & \hbox{$i>k$.}
                   \end{array}
                 \right.
\end{equation}

Therefore for any $\sigma\in S_n$ we have
$(R^{(+)})^{i\sigma(1)}_{i\sigma(1)}\cdots (R^{(+)})^{i\sigma(n)}_{i\sigma(n)}=s(rs)^{n-i}$. Subsequently
for $\langle\varepsilon^{(+)}|=1\otimes\langle\varepsilon^{(+)}_n|\in (V^*)^{\otimes(n+1)}$ and $ |\varepsilon^{(+)}\rangle=1\otimes|\varepsilon^{(+)}_n\rangle\in V^{\otimes(n+1)}$ one has as in (\ref{e:diagonal-block})
\begin{align*}
&\langle\varepsilon^{(+)}|R^{(+)}_{0,n}\cdots R^{(+)}_{0,1}|\varepsilon^{(+)}\rangle\\
&=
\sum_{i=1}^n \frac{s^{\frac{n(n-1)}{2}}}{[n]_{r,s}!}\sum_{\sigma\in S_n}(s^{-1}r)^{l(\sigma)}
(R^{(+)})_{i\sigma(1)}^{i\sigma(1)}\cdots (R^{(+)})_{i\sigma(n)}^{i\sigma(n)}) (e_{ii})_{0}\\
&=
\sum_{i=1}^n s(rs)^{n-i}(e_{ii})_{0}=sM_{0}.
\end{align*}

Applying $A^{(+)}_n=|\varepsilon^{(+)}\rangle\langle\varepsilon^{(+)}|$, we immediately get
\begin{equation}
(1\otimes A^{(+)}_n)R^{(+)}_{0,n}\cdots R^{(+)}_{0,1}(1\otimes A^{(+)}_n)=sM_{0}(1\otimes A^{(+)}_n).
\end{equation}

Now the map
$U(R)\otimes (EndV)^{\otimes n}\rightarrow EndV\otimes (EndV)^{\otimes n} $ given by
$$L^{\pm}_i\mapsto R^{\pm}_{0,i}$$
is an algebra homomorphism, then the identity
$A^{(+)}_nL^{+}_1\cdots L^{+}_n=L^{+}_1\cdots L^{+}_nA^{(+)}_n$ implies that
$$(1\otimes A^{(+)}_n)R^{(+)}_{0,n}\cdots R^{(+)}_{0,1}(1\otimes A^{(+)}_n)=(1\otimes A^{(+)}_n)R^{(+)}_{0,n}\cdots R^{(+)}_{0,1}=sM_{0}(1\otimes A^{(+)}_n).$$
\end{proof}

\begin{theorem}\label{Thm3.6}
The elements
\begin{equation}
c_k=\sum_{\sigma,\sigma'\in S_n}(-s)^{l(\sigma)}(-r)^{-l(\sigma')}l^{+}_{\sigma(1),\sigma'(1)}\cdots l^{+}_{\sigma(k),\sigma'(k)}
l^{-}_{\sigma(k+1),\sigma'(k+1)}\cdots l^{-}_{\sigma(n),\sigma'(n)}
\end{equation}
are quasi-central elements of $U(R)$. Explicitly,
we have
\begin{equation}\label{e:Casimir3}
L^{+}c_k=c_kM^{-1}L^{+}M,
\end{equation}
and
\begin{equation}\label{e:Casimir4}
c_kL^{-}=M'^{-1}L^{-}M'c_k.
\end{equation}

\end{theorem}

\begin{proof} The quantum Casimir elements can be expressed as traces. In fact
\begin{align*}
c_k&=[n]_{r,s}!s^{\frac{n(1-n)}{2}}\langle\varepsilon^{(+)}|L_1^{+}\cdots L_k^{+}L_{k+1}^{-}\cdots L_{n}^{-}|\varepsilon^{(+)}\rangle\\
&=[n]_{r,s}!s^{\frac{n(1-n)}{2}}tr(A^{(+)}_nL_1^{+}\cdots L_k^{+}L_{k+1}^{-}\cdots L_{n}^{-})\\
&=[n]_{r,s}!s^{\frac{n(1-n)}{2}}tr(L_1^{+}\cdots L_k^{+}L_{k+1}^{-}\cdots L_{n}^{-}A^{(+)}_n).
\end{align*}

We take an auxiliary copy $L^{+}$ in the zero position and consider
\begin{equation*}
L^{+}_0L_1^{+}\cdots L_k^{+}L_{k+1}^{-}\cdots L_{n}^{-}A^{(+)}_n.
\end{equation*}

Moving $L^{+}_0$ to the extreme right by the RTT defining relations (\ref{generating relation1}-\ref{generating relation2}),
we have
\begin{align*}
&L^{+}_0L_1^{+}\cdots L_k^{+}L_{k+1}^{-}\cdots L_{n}^{-}A^{(+)}_n\\
=&(R^{(+)}_{0n}\cdots R^{(+)}_{01})^{-1}L_1^{+}\cdots L_k^{+}L_{k+1}^{-}\cdots L_{n}^{-}L^{+}_0R^{(+)}_{0n}\cdots R^{(+)}_{01}A^{(+)}_n\\
=&s(R^{(+)}_{0n}\cdots R^{(+)}_{01})^{-1}L_1^{+}\cdots L_k^{+}L_{k+1}^{-}\cdots L_{n}^{-}A^{(+)}_nL^{+}_0M_0,
\end{align*}
where the last identity uses Lemma \ref{lemma3.5}.

Now we move $A^{(+)}_n$ to the left by Propsition \ref{prop3.4} and use
Lemma \ref{lemma3.5} again:
\begin{align*}
&s(R^{(+)}_{0n}\cdots R^{(+)}_{01})^{-1}L_1^{+}\cdots L_k^{+}L_{k+1}^{-}\cdots L_{n}^{-}A^{(+)}_nL^{+}_0M_0\\
=&s(R^{(+)}_{0n}\cdots R^{(+)}_{01})^{-1}A^{(+)}_nL_1^{+}\cdots L_k^{+}L_{k+1}^{-}\cdots L_{n}^{-}L^{+}_0M_0\\
=&A^{(+)}_nM_0^{-1}L_1^{+}\cdots L_k^{+}L_{k+1}^{-}\cdots L_{n}^{-}L^{+}_0M_0\\
=&A^{(+)}_nL_1^{+}\cdots L_k^{+}L_{k+1}^{-}\cdots L_{n}^{-}M_0^{-1}L^{+}_0M_0.
\end{align*}
Finally taking partial trace $Tr_{1\ldots n}$ in $(\mathbb{C}^n)^{\otimes (n+1)}$
for the right-hand side, we obtain the result that $L^{+}c_k=c_kM^{-1}L^{+}M$.
The identity (\ref{e:Casimir4}) is proved similarly.
\end{proof}

\begin{remark}
When $rs=1$, $c_k$ become the quantum Casimir elements in the center.
\end{remark}

\section{Gauss decomposition of $L^{\pm}$}
In this section we use the Gauss decomposition of $L^{\pm}$ to study the commuting relations
among the quantum Cartan-Weyl generators. Guass decomposition was used by Ding-Frenkel \cite{DF}
to show that the RTT version of the quantum algebras $U_q({\mathfrak gl}_n)$ is isomorphic to
the Drinfeld-Jimbo realization. 

\begin{proposition}
The matrices $L^{\pm}$ can be decomposed as follows.
\begin{equation}\label{Gauss-dec1}
L^+=\begin{pmatrix}
      K^+_1 & &  &  \\
       & K_2^+ &  &  \\
       &  & \ddots &  \\
       &  &  & K_n^+ \\
    \end{pmatrix}
    \begin{pmatrix}
      1 & E_{12} & \ldots & E_{1n} \\
       & 1 & E_{23} & \vdots \\
       &  & \ddots & E_{n-1,n} \\
       &  &  & 1 \\
    \end{pmatrix}
\end{equation}
\begin{equation}
L^-= \begin{pmatrix}\label{Gauss-dec2}
      1 &  &  &  \\
      F_{21} & 1 &  &  \\
      \vdots &  &  \ddots&  \\
      F_{n1} & \ldots & F_{n,n-1}& 1 \\
    \end{pmatrix}
    \begin{pmatrix}
      K^-_1 & &  &  \\
       & K_2^- &  &  \\
       &  & \ddots &  \\
       &  &  & K_n^- \\
    \end{pmatrix}
\end{equation}
where the elements $K_i^{\pm} \, (1\leq i\leq n)$, $ E_{ij}, F_{ji} \, (i<j)$ are exclusively
defined by $L_{ij}^{\pm}$ through
the equations.
\end{proposition}

The following relations are obtained by
the RTT defining relations (\ref{generating relation1}-\ref{generating relation2}) and the Gauss decomposition (\ref{Gauss-dec1}-\ref{Gauss-dec2}).

\begin{proposition}\label{prop4.2}
In $U(R)$ we have
\begin{align}
K^{\pm}_iK^{\pm}_j&=K^{\pm}_jK^{\pm}_i,\\
K^{\pm}_iK^{\mp}_j&=K^{\mp}_jK^{\pm}_i.
\end{align}
\end{proposition}

Next we compute the commutation relations between $K^{\pm}_i$ and $E_{j,j+1}$.
\begin{proposition}\label{prop4.3}
In $U(R)$ we have that
\begin{align}\label{Eq1_prop4.3}
K^{+}_iE_{i,i+1}&=rE_{i,i+1}K^{+}_i,\\
\label{Eq2_prop4.3}
K^{+}_iE_{i-1,i}&=sE_{i-1,i}K^{+}_i,\\
\label{Eq3_prop4.3}
K^{+}_iE_{j,j+1}&=E_{j,j+1}K^{+}_i,\qquad \mbox{for $i\neq j, j+1$};\\
\label{Eq4_prop4.3}
K^{-}_iE_{i,i+1}&=sE_{i,i+1}K^{-}_i,\\
\label{Eq_5_prop4.3}
K^{-}_iE_{i-1,i}&=rE_{i-1,i}K^{-}_i,\\
\label{Eq6_prop4.3}
K^{-}_iE_{j,j+1}&=E_{j,j+1}K^{-}_i, \qquad \mbox{for $i\neq j, j+1$.}
\end{align}
\end{proposition}
\begin{proof} Here we just prove Eqs. (\ref{Eq1_prop4.3}) and (\ref{Eq3_prop4.3}), as the other relations can be obtained similarly.

From the defining relation (\ref{generating relation1}) of $U(R)$ it follows that
$$l_{ii}^{+}l_{i,i+1}^{+}=rl_{i,i+1}^{+}l_{i,i}^{+}.$$
Plugging in the Gauss-decomposition (\ref{Gauss-dec1}), we get Eq. (\ref{Eq1_prop4.3}).

For $i\neq j, j+1$ the defining relation (\ref{generating relation1}) implies that
$$l_{ii}^{+}l_{j,j+1}^{+}=l_{j,j+1}^{+}l_{i,i}^{+},$$
which is equivalent to
$$K_i^{+}K_{j}^{+}E_{j,j+1}=K_{j}^{+}E_{j,j+1}K_{i}^{+}.$$
by using the Gauss decomposition of $L^{+}$. Finally the invertibility of
$K^{\pm}_j$ gives Eq(\ref{Eq3_prop4.3}).
\end{proof}

Similarly we can obtain the commuting relations between $K^{\pm}_i$ and $F_{j+1,j}$.
\begin{proposition}\label{prop4.4}
In $U(R)$ we have that
\begin{align}\label{Eq1_prop4.4}
K^{+}_iF_{i+1,i}&=r^{-1}F_{i+1,i}K^{+}_i,\\
K^{+}_iF_{i,i-1}&=s^{-1}F_{i,i-1}K^{+}_i,\\
K^{+}_iF_{j+1,j}&=E_{j+1,j}K^{+}_i, \qquad \mbox{for $i<j$};\\
K^{-}_iF_{i+1,i}&=s^{-1}F_{i+1,i}K^{-}_i,\\
K^{-}_iF_{i,i-1}&=r^{-1}F_{i,i-1}K^{-}_i,\\
K^{-}_iF_{j+1,j}&=E_{j+1,j}K^{-}_i, \qquad \mbox{for $i<j$}.
\end{align}
\end{proposition}

Now we compute the commuting relations between $E_{i,i+1}$ and $F_{j+1,j}$.
\begin{proposition}\label{prop4.5}
In $U(R)$ we have that
\begin{equation}\label{e:comm1}
[E_{i,i+1},F_{j+1,j}]=\delta_{ij}(r^{-1}-s^{-1})(K_{i+1}^{-}(K_i^{-})^{-1}-K_{i+1}^{+}(K_i^{+})^{-1}).
\end{equation}
\end{proposition}
\begin{proof}
It follows from (\ref{generating relation2}) that
\begin{equation}\label{Eq1_proof3.6}
rsl^{+}_{i,i+1}l^{-}_{i+1,i}+(s-r)l^{+}_{i+1,i+1}l^{-}_{i,i}=
l^{-}_{i+1,i}l^{+}_{i,i+1}+(s-r)l^{-}_{i+1,i+1}l^{+}_{i,i}.
\end{equation}
By Gauss decomposition of $L^{\pm}$
both sides of ({\ref{Eq1_proof3.6}}) can be written as
\begin{align*}
rsK_{i}^{+}E_{i,i+1}F_{i+1,i}K^{-}_{i}+(s-r)K_{i+1}^{+}K_{i}^{-}
=F_{i+1,i}K^{-}_{i}K_{i}^{+}E_{i,i+1}+(s-r)K_{i+1}^{-}K_{i}^{+}.
\end{align*}

 Taking account of Eqs. (\ref{Eq4_prop4.3}) and (\ref{Eq1_prop4.4}), we see that the above is reduced to
 (\ref{e:comm1}).
\end{proof}

The analog of Serre relations is given below for $E_{ij}$.
\begin{proposition}\label{prop4.6}
In $U(R)$ we have that
\begin{align}\label{Eq1_prop4.6}
&E_{i,i+1}^{2}E_{i+1,i+2}-(r+s)E_{i,i+1}E_{i+1,i+2}E_{i,i+1}+rsE_{i+1,i+2}E_{i,i+1}^2=0,\\
\label{Eq2_prop4.6}
&E_{i,i+1}E_{i+1,i+2}^2-(r+s)E_{i+1,i+2}E_{i,i+1}E_{i+1,i+2}+rsE_{i+1,i+2}^{2}E_{i,i+1}=0,\\
&E_{i,i+1}E_{j,j+1}=E_{j,j+1}E_{i,i+1}, \qquad\mbox{if $|i-j|>1$}.
\end{align}
\end{proposition}
\begin{proof}
We first prove the following commutation relation.
\begin{equation}\label{Eq1_lemma3.7}
E_{i+1,i+2}E_{i,i+1}=s^{-1}E_{i,i+1}E_{i+1,i+2}+(r^{-1}-s^{-1})E_{i,i+2}.
\end{equation}

In fact, from the defining relation (\ref{generating relation1}) it follows that
\begin{equation*}
rsl^{+}_{i,i+1}l^{+}_{i+1,i+2}+(s-r)l^{+}_{i+1,i+1}l^{+}_{i,i+2}=
rsl^{+}_{i+1,i+2}l^{+}_{i,i+1},
\end{equation*}
Plugging in the Gauss decomposition, this becomes
\begin{equation}\label{Eq2_proof3.7}
rsK_{i}^{+}E_{i,i+1}K^{+}_{i+1}E_{i+1,i+2}+(s-r)K_{i+1}^{+}K_{i}^{+}E_{i,i+2}=rsK^{+}_{i+1}E_{i+1,i+2}K_{i}^{+}E_{i,i+1}.
\end{equation}
Then Eq. (\ref{Eq1_lemma3.7}) is obtained by using Eqs. (\ref{Eq2_prop4.3}) and (\ref{Eq3_prop4.3}).

Multiplying $E_{i,i+1}$ from the left and the right of (\ref{Eq1_lemma3.7}), we have that
\begin{align*}
E_{i,i+1}^{2}E_{i+1,i+2}&=sE_{i,i+1}E_{i+1,i+2}E_{i,i+1}-(r^{-1}s-1)E_{i,i+1}E_{i,i+2},\\
rsE_{i+1,i+2}E_{i,i+1}^{2}&=rE_{i,i+1}E_{i+1,i+2}E_{i,i+1}+(s-r)E_{i,i+2}E_{i,i+1}.
\end{align*}
On the other hand, the defining relation (\ref{generating relation1}) also gives that
\begin{equation*}
E_{i,i+1}E_{i,i+2}=rE_{i,i+2}E_{i,i+1}.
\end{equation*}
Then (\ref{Eq1_prop4.6}) immediately follows. Eq. (\ref{Eq2_prop4.6}) is proved similarly.
\end{proof}

In the same way we obtain the Serre relations for $F_{ij}$.
\begin{proposition}\label{prop4.7}
In $U(R)$ we have
\begin{align}
&F_{i+1,i}^{2}F_{i+2,i+1}-(r^{-1}+s^{-1})F_{i+1,i}F_{i+2,i+1}F_{i+1,i}+r^{-1}s^{-1}F_{i+2,i+1}F_{i+1,i}^2=0,\\
&F_{i+1,i}F_{i+2,i+1}^2-(r^{-1}+s^{-1})F_{i+2,i+1}F_{i+1,i}F_{i+2,i+1}+r^{-1}s^{-1}F_{i+2,i+1}^{2}F_{i+1,i}=0,\\
&F_{i+1,i}F_{j+1,j}=F_{j+1,j}F_{i+1,i}, \qquad\mbox{if $|i-j|>1$}.
\end{align}
\end{proposition}

\section{Isomorphism between the quantum group $U_{r,s}(\mathfrak{gl}_n)$ and $U(R)$}
The current version of two-parameter
 quantum group was given by Benkart and Witherspoon \cite{BW1} in terms of Chevalley generators and
 Serre relations
in connection with the down-up algebras.
In \cite{BW2,BW3} they further developed the representation theory
of the two-parameter quantum general and special linear algebras
and constructed the corresponding $R$-matrix. We now identify their version
with our FRT version given in earlier sections.

Let $\epsilon_1$, $\epsilon_2, \ldots, \epsilon_n$ be the orthonormal basis of the Euclidean space $\mathbb{C}^n$ with inner product
$\langle \ \ , \ \ \rangle$. Let $\Pi=\{\alpha_j=\epsilon_j-\epsilon_{j+1}| j=1,2,...n-1\}$ and $\Phi=\{\epsilon_i-\epsilon_j|1\leq i\neq j\leq n\}$.
Then $\Phi$ realizes the root system of type $A_{n-1}$ with $\Pi$ a base of simple roots.

\begin{definition} The algebra
$U_{r,s}(\mathfrak{gl}_n)$ is a unital associated algebra over $\mathbb{C}$ generated by
$e_j$, $f_j$, $(1\leq j< n)$, and $a_i^{\pm 1}$, $b_i^{\pm 1}$ $(1\leq i\leq n)$ subject to the following relations.
\begin{description}
  \item[R1] $a_i^{\pm 1}$, $b_j^{\pm 1}$ $(1\leq i\leq n)$ commutate with each other and
  $a_ia_i^{-1}=b_ib_i^{-1}=1$;

  \item[R2] $a_ie_j=r^{\langle \epsilon_i,\alpha_j\rangle}e_ja_i$, and $a_if_j=r^{-\langle \epsilon_i,\alpha_j\rangle}f_ja_i$;
  \item[R3] $b_ie_j=s^{\langle \epsilon_i,\alpha_j\rangle}e_jb_i$, and $b_if_j=s^{-\langle \epsilon_i,\alpha_j\rangle}f_jb_i$;

  \item[R4] $[e_i,f_j]=\frac{\delta_{ij}}{r-s}(a_ib_{i+1}-a_{i+1}b_i)$;
  \item[R5] $[e_i,e_j]=[f_i,f_j]=0$ if $|i-j|>1$;
  \item[R6] $e_i^2e_{i+1}-(r+s)e_ie_{i+1}e_i+rse_{i+1}e_i^2=0$,\\
            $e_{i+1}^2e_{i}-(r+s)e_{i+1}e_{i}e_{i+1}+rse_{i}e_{i+1}^2=0$,
  \item[R7] $f_i^2f_{i+1}-(r^{-1}+s^{-1})f_if_{i+1}f_i+r^{-1}s^{-1}f_{i+1}f_i^2=0$;\\
            $f_{i+1}^2f_{i}-(r^{-1}+s^{-1})f_{i+1}f_{i}f_{i+1}+r^{-1}s^{-1}f_{i}f_{i+1}^2=0$.

\end{description}
\end{definition}

The algebra $U_{r,s}(\mathfrak{gl}_n)$ is a Hopf algebra such that
$a_i^{\pm 1}$, $b_i^{\pm 1}$ are the group-like
elements and the remaining Hopf algebra structure is given by
\begin{equation*}
\delta(e_i)=e_i\otimes 1+\omega_i\otimes e_i,
\end{equation*}
\begin{equation*}
\delta(f_i)=f_i\otimes \omega'_i+1\otimes f_i,
\end{equation*}
\begin{equation*}
\varepsilon(e_i)=\varepsilon(f_i)=0,
\end{equation*}
\begin{equation*}
S(e_i)=-\omega^{-1}_ie_i,
\end{equation*}
\begin{equation*}
S(f_i)=-f_i\omega'^{-1}_i.
\end{equation*}
\begin{remark}
When $r=q=s^{-1}$, the Hopf algebra $U_{r,s}(\mathfrak{gl}_n)$ modulo the ideal generated by
$b_i-a_i^{-1}$, $1\leq i\leq n$ is isomorphic to $U_q(\mathfrak{gl}_n)$.
\end{remark}

Analogous with the one-parameter case, we have the following theorem.
\begin{theorem}
The mapping $\psi: U_{r,s}(\mathfrak{gl}_n)\rightarrow U(R)$ given by
\begin{align}
e_i\mapsto \frac{r}{r-s}E_{i,i+1}, f_i\mapsto \frac{s}{s-r}F_{i+1,i}, a_1\mapsto K_1^{+}, b_1\mapsto K_1^{-},
\end{align}
and
\begin{align}
a_i&\mapsto K_i^{+}\prod_{l=1}^{i-1}(K_{i-l}^{+}K_{i-l}^{-})^{(-1)^l},\\
b_i&\mapsto K_i^{-}\prod_{l=1}^{i-1}(K_{i-l}^{+}K_{i-l}^{-})^{(-1)^l},
\end{align}
is an isomorphism.
\end{theorem}

\begin{proof}
By Propositions \ref{prop4.2}, \ref{prop4.6} and \ref{prop4.7}, relations R1, R5, R6, and R7 hold.
We only need to check relations R2, R3, and R4.

First let consider R2.
It follows from Proposition \ref{prop4.3} that
\begin{equation}
\begin{aligned}
K_i^{+}\prod_{l=1}^{i-1}(K_{i-l}^{+}K_{i-l}^{-})^{(-1)^l}
E_{i,i+1}&=K_i^{+}{E_{i,i+1}}\prod_{l=1}^{i-1}(K_{i-l}^{+}K_{i-l}^{-})^{(-1)^l}\\
&=rE_{i,i+1}K_i^{+}\prod_{l=1}^{i-1}(K_{i-l}^{+}K_{i-l}^{-})^{(-1)^l}.
\end{aligned}
\end{equation}
\begin{equation}
\begin{aligned}
&K_i^{+}\prod_{l=1}^{i-1}(K_{i-l}^{+}K_{i-l}^{-})^{(-1)^l}
E_{i-1,i}\\
=&K_i^{+}(K_{i-1}^{+})^{-1}(K_{i-1}^{-})^{-1}E_{i-1,i}\prod_{l=2}^{i-1}(K_{i-l}^{+}K_{i-l}^{-})^{(-1)^l}\\
=&r^{-1}s^{-1}K_i^{+}E_{i-1,i}(K_{i-1}^{+})^{-1}(K_{i-1}^{-})^{-1}\prod_{l=2}^{i-1}(K_{i-l}^{+}K_{i-l}^{-})^{(-1)^l}\\
=&r^{-1}E_{i-1,i}K_i^{+}\prod_{l=1}^{i-1}(K_{i-l}^{+}K_{i-l}^{-})^{(-1)^l}.
\end{aligned}
\end{equation}
When $j\neq i, i-1$, we have
\begin{equation}
\begin{aligned}
K_i^{+}\prod_{l=1}^{i-1}(K_{i-l}^{+}K_{i-l}^{-})^{(-1)^l}
E_{j,j+1}=E_{j,j+1}K_i^{+}\prod_{l=1}^{i-1}(K_{i-l}^{+}K_{i-l}^{-})^{(-1)^l}.
\end{aligned}
\end{equation}
Similarly we also have that
\begin{align*}
K_i^{+}\prod_{l=1}^{i-1}(K_{i-l}^{+}K_{i-l}^{-})^{(-1)^l}
F_{i+1,i}&=r^{-1}F_{i+1,i}K_i^{+}\prod_{l=1}^{i-1}(K_{i-l}^{+}K_{i-l}^{-})^{(-1)^l},\\
K_i^{+}\prod_{l=1}^{i-1}(K_{i-l}^{+}K_{i-l}^{-})^{(-1)^l}
F_{i,i-1}&=rF_{i,i-1}K_i^{+}\prod_{l=1}^{i-1}(K_{i-l}^{+}K_{i-l}^{-})^{(-1)^l},
\end{align*}
and for $j\neq i,i-1$
\begin{equation*}
K_i^{+}\prod_{l=1}^{i-1}(K_{i-l}^{+}K_{i-l}^{-})^{(-1)^l}
F_{j+1,i}=F_{j+1,j}K_i^{+}\prod_{l=1}^{i-1}(K_{i-l}^{+}K_{i-l}^{-})^{(-1)^l}.
\end{equation*}
Then relation R2 is satisfied. Relation R3 is proved similarly.

Next we consider relation R4. It from Proposition \ref{prop4.5} that
$$\frac{r}{r-s}\frac{s}{s-r}[E_{i,i+1},F_{i+1,i}]=\frac{1}{r-s}(K_{i+1}^{-}(K_i^{-})^{-1}-K_{i+1}^{+}(K_i^{+})^{-1}).$$
It is easy to check that
$$K_i^{+}\prod_{l=1}^{i-1}(K_{i-l}^{+}K_{i-l}^{-})^{(-1)^l}K_{i+1}^{-}\prod_{l=1}^{i}(K_{i+1-l}^{+}K_{i+1-l}^{-})^{(-1)^l}=K_{i+1}^{-}(K_i^{-})^{-1},$$
$$K_{i+1}^{+}\prod_{l=1}^{i}(K_{i+1-l}^{+}K_{i+1-l}^{-})^{(-1)^l}K_i^{-}\prod_{l=1}^{i-1}(K_{i-l}^{+}K_{i-l}^{-})^{(-1)^l}=K_{i+1}^{+}(K_i^{+})^{-1}).$$
So R4 holds. Note that all the $E_{ij}$, $F_{ji}$, $i<j$, can be generated by $E_{i,i+1}$, $F_{i+1,i}$. Therefore
$\psi$ is a surjective homomorphism.

The injectivity can be proved verbatim as in \cite{DF} for the one-parameter case.
\end{proof}

\bigskip

\centerline{\bf Acknowledgments}

 NJ gratefully acknowledges the support of
Humboldt Foundation, MPI-Leipzig, Simons Foundation
grant 198129, and NSFC grant  during this work.

\bibliographystyle{amsalpha}

\begin{thebibliography}{99}

\bibitem{BR} C. Briot and E. Ragoucy, RTT presentation of finite W-algebras,
J. Phys. A 34 (2001), 7287-7310.
\bibitem{CP} V. Chari and A. Pressley, {\em A guide to Quantum
Groups}, Cambridge Univ. Press, Cambridge, 1994.
\bibitem{BW1} G. Benkart and S. Witherspoon, {\em A Hopf structure for down-up algebras}, Math. Z. 238 (2001), 523-553.
\bibitem{BW2} G. Benkart and S. Witherspoon, {\em Two-parameter quantum groups and Drinfeld doubles}, Algebr. Represent. Theory 7 (2004), 261--286.
\bibitem{BW3} G. Benkart and S. Witherspoon, {\em Representations of two-parameter quantum groups and Schur-Weyl duality}, Hopf algebras, Lecture Notes in pure and Appl. Math., 237 (2004), 65-92.
\bibitem{BGH} N. Bergeron, Y. Gao and N. Hu, {\em Drinfel'd doubles and Lusztig's symmetries of twoparameter
quantum groups}. J. Algebra 301 (2006), 378--405.
\bibitem{DF} J. Ding and I. B. Frenkel, {\em Isomorphism of two realizations of quantum affine algebra $U_q (\widehat{\mathfrak{gl}(n)})$}, Commun. Math. Phys. 156 (1993), 277--300.
\bibitem{DD} R. Dipper and S. Donkin, {\em Quantum $GL_n$}.
Proc. London Math. Soc. (3) 63 (1991), no. 1, 165--211.
\bibitem{D1} V. Drinfeld, {\em Hopf algebras and the quantum Yang-Baxter
equation}, Soviet Math. Dokl. 32 (1985), 254--258.
\bibitem{D3} V. Drinfeld, {\em Quantum Group}, Proc. ICM, Vol. 1, 2 (Berkeley, Calif., 1986), 798--820, Amer. Math. Soc., Providence, RI, 1987.
\bibitem{DPW} J. Du, B. Parshall and J. Wang, {\em Two-parameter quantum linear groups and the hyperbolic invariance of q-Schur algebras}. J. London Math. Soc. (2) 44 (1991), 420--436.
\bibitem{FRT} L. Faddeev, N. Reshetikhin and L. Takhtadzhyan, {\em Quantization
of Lie groups and Lie algebras}, Leningrad Math. J. 1 (1990),
193--225.
\bibitem{HIOPT} L.~K. Hadjiivanov, A.~P. Isaev, O.~V. Ogievetsky, P.~N. Pyatov and I.~T. Todorov,
{\em Hecke algebraic properties of dynamical R-matrices. Application to related quantum matrix algebras},
J. Math. Phys. 40 (1999), 427--448.
\bibitem{HRZ} N. Hu, M. Rosso and H. Zhang, {\em Two-parameter quantum affine algebra $U_{r,s}(\widehat{\mathfrak{sl}}_n)$, Drinfeld realization and 
    quantum affine Lyndon basis}, Commun. Math. Phys. 278 (2008), 453--486.
\bibitem{J} M. Jimbo, {\em A q-difference analogue of $U(g)$ and the Yang-Baxter equation}, Lett. Math. Phys. 10 (1985), 63--69.
\bibitem{JL} N. Jing and M. Liu, {\em Fusion procedure for the two-parameter quantum algebra $U_{r,s}(\mathfrak{sl}_n)$},
arXiv:1402.3665.
\bibitem{JZ} N. Jing and H. Zhang, {\em Fermionic realization of two-parameter quantum
affine algebra $U_{r,s}(\widehat{{sl}_n})$}, Lett. Math. Phys. 89 (2009), no. 2, 159--170.
\bibitem{M} A. Molev, {\em Yangians and classical Lie algebras},
Math. Surv. and Monograph, 143. AMS, Providence, RI, 2007.
\bibitem{MNO} A. Molev, M. Nazarov and G. Olshanskii, {\em Yangians and classical
Lie algebras}, Russian Math. Surveys 51 (1996), 205--282.
\bibitem{T} M. Takeuchi, {\em A two-parameter quantization of $GL(n)$}, Proc. Japan. Acad. 66 Ser. A (1990), 112--114.
\end{thebibliography}

\end{document}